\documentclass{cmslatex}
\usepackage[paperwidth=7in, paperheight=10in, margin=.875in]{geometry}
\usepackage{marktext} 
\usepackage{amsfonts,amssymb,amsmath}
\usepackage[colorlinks,linkcolor=red,anchorcolor=green,citecolor=blue]{hyperref}
\usepackage{svmacro}
\usepackage[utf8]{inputenc}

\renewcommand{\theequation}{\arabic{section}.\arabic{equation}}

\date{\today}

\begin{document}
\title{Large time behavior for a Hamilton-Jacobi equation in a critical Coagulation-Fragmentation model}

\author{Hiroyoshi Mitake\thanks{Graduate School of Mathematical Sciences,  University of
    Tokyo, 3-8-1 Komaba, Meguro-ku, Tokyo, 153-8914, Japan} (\href{mailto:mitake@ms.u-tokyo.ac.jp}{mitake@ms.u-tokyo.ac.jp})
        \and
        Hung V. Tran\thanks{Department of Mathematics,
        University of Wisconsin Madison,
        Van Vleck Hall, 480 Lincoln Drive, Madison, WI 53706} (\href{mailto:hung@math.wisc.edu}{hung@math.wisc.edu})
        \and
        Truong-Son Van\thanks{Department of Mathematical Sciences and
        Center for Nonlinear Analysis,
        Carnegie Mellon University, Pittsburgh, PA 15213 (\href{mailto:sonv@andrew.cmu.edu}{sonv@andrew.cmu.edu})}}

\pagestyle{myheadings} \markboth{Large time behavior for a H-J eq. in a critical C-F eq.}{Mitake, Tran and Van} \maketitle

\begin{abstract}
We study the large time behavior of the sublinear viscosity solution to a singular Hamilton-Jacobi equation that appears in a  critical Coagulation-Fragmentation model with multiplicative coagulation and constant fragmentation kernels. Our results include complete characterizations of stationary solutions and optimal conditions to guarantee large time convergence. In particular, we obtain convergence results under certain natural conditions on the initial data, and a nonconvergence result when such conditions fail.
\end{abstract}


\smallskip

\section{Introduction}
The Coagulation-Fragmentation equation (C-F) is an integrodifferential equation that describes the evolution of distribution of objects via simple mechanisms of coalescence and breakage. In its strong form, the continuous Coagulation-Fragmentation equation reads as follows
\begin{equation}\label{eq:CF}
    \begin{cases}
          \partial_t \rho(s,t) = Q_c(\rho)(s,t) + Q_f(\rho)(s,t) \quad &\text{ in } (0,\infty)\times (0,\infty)\,,\\
          \rho(s,0)=\rho_0(s) \quad &\text{ on } [0,\infty)\,.
    \end{cases}
\end{equation}
Here, $\rho(s,t) \geq 0$ is the density of clusters of particles of size $s \geq 0$ at time $t\geq 0$.
The coagulation term  $Q_c$ and the fragmentation term $Q_f$ are given by
\[
Q_c(\rho)(s,t)=\frac{1}{2} \int_0^s a(y,s-y)\rho(y,t) \rho(s-y,t)\,dy - \rho(s,t) \int_0^\infty a(s,y) \rho(y,t)\,dy\,,
\]
and
\[
Q_f(\rho)(s,t)=-\frac{1}{2} \rho(s,t) \int_0^s b(s-y,y)\,dy + \int_0^\infty b(s,y) \rho(y+s,t)\,dy\,.
\]
The coagulation kernel $a$ and the fragmentation kernel $b$ are non-negative and symmetric functions on $(0,\infty)^2$.

Although the equation has a history of over a hundred years and despite the works of many mathematicians, there are still a lot of mathematical mysteries about it. In particular, the most basic question about wellposedness has not been addressed satisfactorily and is an active research area.
For more historical contexts and surveys of the field, we refer the reader to the following works~\cite{Aldous99, Costa15, LAURENCOT2019-2, LAURENCOT2019}.

In this work, we restrict our attention to the multiplicative coagulation and constant fragmentation kernels, that is,
\begin{equation}\tag{A} \label{A:kernel}
    a(s,\tilde s) = s\tilde s \quad \text{ and } \quad b(s,\tilde s) = 1 \quad \text{ for all } s,\tilde s > 0.
\end{equation}
This is a so-called critical case (among other more complicated ones), where the existence of mass-conserving solutions depends on the initial data. 
Despite multiple efforts using different approaches, the wellposedness theory for this particular case has not been fully established. 
In particular, letting $m$ be the first moment of the initial data $\rho_0$,
using the moment bound method in~\cite{LAURENCOT2019-2}, Lauren\c{c}ot, under certain assumptions on initial moments, established existence and uniqueness of mass-conserving solutions for $m \in \left(0, \frac{1}{4 \log 2}\right)$. By studying the viscosity solution to a singular Hamilton-Jacobi equation that results from applying the Bernstein transform to equation~\eqref{eq:CF}, the second and third authors  established existence and uniqueness of mass-conserving measure valued solutions for $m\in \left(0,\frac{1}{2}\right)$.  
This approach was initiated in~\cite{TranVan2019}, inspired by the works of Menon and Pego, who pioneered the study of the Smoluchowski equation (C-F with pure coagulation) via Bernstein transform~\cite{MenonPego04, MenonPego05, MenonPego06, MenonPego07, MenonPego08}.

Non-existence of mass-conserving solutions for $m>1$ were established first in~\cite{BanasiakLambEA19} by the moment bound method and confirmed again with minimal assumptions in~\cite{TranVan2019} by studying the corresponding Hamilton-Jacobi equation. 
Furthermore, while uniqueness of mass-conserving solutions for $m \in [\frac{1}{2}, 1]$ was established in~\cite{TranVan2019},
the existence question remains an outstanding open problem.

\smallskip

Here, we will not discuss the wellposedness theory but, rather, focus on studying the dynamics of solutions. 
Specifically, we are interested in the long-time behavior of the solutions when $m=1$. For $m \in (0,1)$, it was shown in~\cite{TranVan2019} that all solutions will turn to dust (particles of size zero) as $t \to \infty$, i.e., $\lim_{t\to \infty} \rho(s,t) = m \delta_0$. 
The difficulty for the case $m=1$ lies in the fact that there are infinitely many stationary solutions. 
This was observed by Lauren\c{c}ot via private communications and recorded in~\cite{TranVan2019}. 
Therefore, full characterizations of stationary solutions are needed. It is also unclear from the Hamilton-Jacobi equation point of view that the viscosity solution converges to a stationary solution as $t\to\infty$. 
To address these questions, we need to study more deeply the viscosity solution to the aforementioned Hamilton-Jacobi equation. 

\subsection{Bernstein transform and Hamilton-Jacobi equation}
For a nonnegative measure $\mu$ on $[0,\infty)$ such that $\int_0^\infty \min\{1,s\}\, \mu(ds)<\infty$,
its Bernstein transform is defined by the following integral
\begin{equation*}
    \mathfrak{B}[\mu](x) \defeq \int_0^\infty (1- e^{-sx}) \mu(ds) \,. 
\end{equation*}

Writing the equation~\eqref{eq:CF} under assumption~\eqref{A:kernel} in its weak form, we have that for every test function $\phi \in BC([0,\infty))\cap \Lip([0,\infty))$ such that $\phi(0)=0$,
\begin{equation}\label{weak sln}
  \begin{split}
    \frac{d}{dt} \int_0^\infty \phi(s) \rho(s,t) \, ds &= \frac{1}{2} \int_0^\infty \int_0^\infty (\phi(s + \hat s) - \phi(s) - \phi(\hat s)) s \rho(s,t) \hat s \rho(\hat s,t) \, d\hat s ds \\ 
    & \quad -\frac{1}{2} \int_0^\infty \paren[\Big]{ \int_0^s (\phi(s) - \phi(\hat s) - \phi(s-\hat s)) \,d\hat s  } \rho(s,t) \, ds \,.
  \end{split}
\end{equation}

Letting $\phi^x(s) = 1 - e^{-sx}$ be a test function in the above for each $x \geq 0$, and denote 
\begin{equation*}
    F(x,t) \defeq \mathfrak{B}[\rho](x,t) \quad \text{ and } \quad  F_0(x) \defeq \mathfrak{B}[\rho_0](x) \,.
\end{equation*}
Here, $\rho_0 \geq 0$ is the given initial data.
If conservation of mass (first moment) holds, that is,
 \begin{equation*}
     m_1(t) = \int_0^\infty s \rho(s,t)\, ds=\int_0^\infty s \rho_0(s)\, ds= m
 \end{equation*}
 for all $t\geq 0$ for some given $m>0$, then we have the following equation (see \autoref{appendix} for a derivation)
\begin{equation}\label{HJ}
\begin{cases}
    \partial_t F + \frac{1}{2}(\partial_x F - m)(\partial_x F - m -1) + \frac{F}{x} - 1 = 0 \qquad &\text{in $(0,\infty)^2$},\\
    0 \leq F(x,t) \leq mx \qquad &\text{on $[0,\infty)^2$},\\
    F(x,0)=F_0(x) \qquad &\text{on $[0,\infty)$}.
\end{cases}
\end{equation}

 We focus on the case that $m=1$ in this paper.
 Thus, the main equation of interests is
 \begin{equation}\label{HJ-1}
\begin{cases}
    \partial_t F + \frac{1}{2}(\partial_x F - 1)(\partial_x F - 2) + \frac{F}{x} - 1 = 0 \qquad &\text{in $(0,\infty)^2$},\\
    0 \leq F(x,t) \leq x \qquad &\text{on $[0,\infty)^2$},\\
    F(x,0)=F_0(x) \qquad &\text{on $[0,\infty)$}.
\end{cases}
\end{equation}
Appropriate conditions on initial data $F_0$ will be specified in the next subsection.
Large time behavior of \eqref{HJ-1} has not been studied in the literature, and this was left as an open problem in \cite{TranVan2019}.
We are always concerned here with viscosity solutions of first-order Hamilton-Jacobi equations, and the adjective ``viscosity'' is omitted henceforth.

\subsection{Main results}
In this subsection, we give an outline of our findings.
For each fixed $x\geq 0$, $F(x,t)$ is bounded for all $t\geq 0$ as $0 \leq F(x,t)\leq x$. Therefore, for stationary solutions, it is reasonable to impose that $\partial_t F=0$, and hence, \eqref{HJ-1} becomes
\begin{equation}\label{HJ-s}
\begin{cases}
    \frac{1}{2}(\partial_x F - 1)(\partial_x F - 2) + \frac{F}{x} - 1 = 0 &\qquad \text{in $(0,\infty)$},\\
    0 \leq F(x) \leq x &\qquad \text{on $[0,\infty)$}.
\end{cases}
\end{equation}
Our first goal is to characterize all continuous sublinear viscosity solutions to \eqref{HJ-s}.

\smallskip

\begin{theorem}\label{thm:characterization3}
Let $F \in C([0,\infty))$ be a sublinear viscosity solution to \eqref{HJ-s}.
Then, either $F \equiv 0$ or there exists $c>0$ such that
\[
F(x) = \frac{1}{c} \bar F(cx) \qquad \text{ for all } x\geq 0.
\]
Here, $\bar F: [0,\infty) \to [0,\infty)$ is such that $\bar F(0)=0$, and
\[
\partial_x \bar F(x) = \frac{1}{\sqrt{x}} \brak[\Bigg]{ \paren[\Big]{ \frac{\sqrt{x+x_0}+\sqrt{x}}{2} }^{1/3} 
- \paren[\Big]{ \frac{\sqrt{x+x_0}-\sqrt{x}}{2}}^{1/3}  }
\]
for $x_0=\frac{4}{27}$.
\end{theorem}

\smallskip

\begin{remark}
    Note that we do not require any differentiability of $F$ a~priori in the above theorem.
\end{remark}

Next, we study the large time behavior of the viscosity solution to \eqref{HJ-1}.
Large time behavior for Hamilton-Jacobi equations is a rich and very active subject.
We refer the readers to \cite{Fathi1998,BarlesSouganidis2000a, DaviniSiconolfi2006, CagnettiGomesMitakeTran2015} in the periodic setting,
and \cite{BarlesSouganidis2000, BarlesRoquejoffre2006, IchiharaIshii2008, Ishii2009,GigaMitakeTran2019} in noncompact settings for some representative results.
It is worth emphasizing that \eqref{HJ-1} is in a noncompact setting, and is not of the type that was studied earlier in the literature because of the singular term $\frac{F}{x}$.

For initial data $F_0$, we assume first that
\begin{equation}\label{con:A1}
\begin{cases}
    0 \leq F_0(x) \leq x, \ F_0 \text{ is sublinear},\\
    0 \leq \partial_x F_0 \leq 1.
\end{cases}
\end{equation}
The above condition \eqref{con:A1} holds true when $F_0$ is the Bernstein transform of $\rho_0=\rho(\cdot,0)$, whose first moment is $1$.
Indeed,
\[
0\leq F_0(x)=\int_0^\infty (1- e^{-xs})\rho(s,0)\,ds \leq
\int_0^\infty x s \rho(s,0)\,ds = x\, ,
\]
and, by the dominated convergence theorem,
\[
\lim_{x\to \infty} \frac{F_0(x)}{x}=\lim_{x\to \infty}\int_0^\infty \frac{1- e^{-xs}}{x}\rho(s,0)\,ds =0\,.
\]
Besides,
\[
0\leq \partial_x F_0(x)=\int_0^\infty s e^{-xs}\rho(s,0)\,ds \leq 1\, ,
\]
and $\partial_x F_0(0)=1$.
An important point is that we do not need to require conditions on the higher derivatives of $F_0$ here in order to study large time behavior of \eqref{HJ-1} although if $F_0(x) = \mathfrak{B}[\rho_0](x)$, then $F_0$ is smooth.

\medskip

There are three regimes of the initial data $F_0$ to be considered: subcritical, critical, and supercritical. 
We say that the initial data $F_0$ of equation~\eqref{HJ-1} is 

\smallskip

\begin{enumerate}
    \item subcritical if 
        \begin{equation}\label{con:A2}
            \lim_{x \to \infty}\frac{F_0(x)}{x^{2/3}}=0 \,;
        \end{equation}
    \item critical if there exists $\delta >0$ such that
        \begin{equation}\label{con:A4}
            \lim_{x\to \infty}\frac{F_0(x)}{x^{2/3}} =\delta  \,;
        \end{equation}
     \item supercritical if 
        \begin{equation}\label{con:A3}
            \lim_{x\to \infty}\frac{F_0(x)}{x^{2/3}} =  \infty \,.
        \end{equation}
\end{enumerate}
This characterization comes from the observation that the stationary solution $\bar F$ behaves like $O(x^{2/3})$ as $x \to \infty$.
Here are our large time behavior results corresponding to the three different regimes.

\smallskip

\begin{theorem}\label{thm:large time2}
Assume \eqref{con:A1} and \eqref{con:A2}.
Let $F$ be the unique viscosity solution to \eqref{HJ-1}.
Then, as $t\to \infty$, 
\[
F(x,t) \to 0 \qquad \text{ locally uniformly for $x\in[0,\infty)$.} 
\]
\end{theorem}

\begin{theorem}\label{thm:large time4}
Assume \eqref{con:A1} and  \eqref{con:A4} for some given $\delta>0$.
Let $F$ be the unique viscosity solution to \eqref{HJ-1}, and $c=\frac{27}{8\delta^3}$.
Then, as $t\to \infty$, 
\[
F(x,t) \to \frac{1}{c} \bar F(cx) \qquad \text{ locally uniformly for $x\in [0,\infty)$,} 
\]
where $\bar F$ is given in Theorem \ref{thm:characterization3}.
\end{theorem}

\smallskip

\begin{theorem}\label{thm:large time3}
Assume~\eqref{con:A1} and~\eqref{con:A3}.
Let $F$ be the unique viscosity solution to \eqref{HJ-1}.
Then, as $t\to \infty$, 
\[
F(x,t) \to x \qquad \text{ locally uniformly for $x\in [0,\infty)$.} 
\]
\end{theorem}

We now show that the requirements on the initial condition to get large time behavior results in Theorems \ref{thm:large time2}--\ref{thm:large time3} are essential.
In other words, if \eqref{con:A2}--\eqref{con:A3} do not hold, that is,
\begin{equation}\label{lim-not-exist}
0<\liminf_{x\to \infty} \frac{F_0(x)}{x^{2/3}} <\limsup_{x\to \infty} \frac{F_0(x)}{x^{2/3}} <\infty,
\end{equation}
then large time behavior might fail.

\smallskip

\begin{theorem}\label{thm:nonconvergence}
    Let $F$ be the unique viscosity solution to~\eqref{HJ-1}. There exists $F_0 \in \Lip([0,\infty))$ that satisfies~\eqref{con:A1} in the a.e. sense and \eqref{lim-not-exist} such that for some $x_0 > 0$, $\lim_{t\to \infty} F(x_0,t)$ does not exist.
\end{theorem}

\smallskip

See \cite{BarlesSouganidis2000, Ishii2009} for some related results.

\subsection*{Organization of the paper}
The proof of Theorem~\ref{thm:characterization3} is given in  Section~\ref{sec:characterization}, which also contains other characterization results of viscosity solutions to \eqref{HJ-s}.
The proofs of Theorems~\ref{thm:large time2}--\ref{thm:large time3} are given in Section~\ref{sec:large-time}.
In Section~\ref{sec:nonconvergence}, we give the proof of Theorem~\ref{thm:nonconvergence}, where the initial condition $F_0$ is constructed explicitly.

\section{Characterization of all stationary sublinear solutions}\label{sec:characterization}

This section is devoted to prove Theorem~\ref{thm:characterization3}.
In order to do so, we need some preparation.

\smallskip

\begin{proposition}\label{prop:characterization1}
Let $F$ be a solution to \eqref{HJ-s} such that F satisfies
\begin{equation}\label{eq:A1}
\begin{cases}
F \in C^2((0,\infty)) \cap C([0,\infty)), \\
0 < \partial_x F(x) <1 \qquad \text{ for all $x\in (0,\infty)$.}
\end{cases}
\end{equation}
Then, there exists $c>0$ such that
\[
F(x) = \frac{1}{c} \bar F(cx) \qquad \text{ for all } x\geq 0.
\]
Here, $\bar F: [0,\infty) \to [0,\infty)$ is such that $\bar F(0)=0$, and
\[
\partial_x \bar F(x) \defeq 
\frac{1}{\sqrt{x}} \brak[\Bigg]{ \paren[\Big]{\frac{\sqrt{x+x_0} + \sqrt{x}}{2}  }^{1/3} - \paren[\Big]{\frac{\sqrt{x+x_0} - \sqrt{x}}{2}  }^{1/3} } 
\]
for $x_0=\frac{4}{27}$.
\end{proposition}

\smallskip

This proposition gives more or less a similar conclusion as that in Theorem \ref{thm:characterization3} but it requires a more restrictive condition \eqref{eq:A1} on $F$.

\smallskip

\begin{proof}
Letting $G = 1 - \partial_x F$ and from~\eqref{HJ-s}, we get 
\begin{equation*}
    \frac{1}{2} G(G+1) - \frac{1}{x} \int_0^x G(y) \, dy = 0 \,.
\end{equation*}
Therefore, for $x>0$,
\begin{equation*}
    \frac{1}{2} x G(G+1) - \int_0^x G(y) \, dy = 0 \,.
\end{equation*}
Differentiating in $x$,
\begin{equation*}
    \frac{1}{2} G(G+1) + \frac{1}{2} x( 2G\partial_x G + \partial_x G) - G = 0 \,, 
\end{equation*}
which, after rearranging terms, gives 
\begin{equation*}
    \frac{1}{x} = \frac{\partial_x G}{G} + \frac{3 \partial_x G}{1-G} \,.
\end{equation*}
Integrating this equality, we get
\begin{equation} \label{e:transcendental}
    \frac{G(x)}{(1- G(x))^3} = cx \,,
\end{equation}
for $x >0$ and some fixed constant $c>0$.

\medskip 

Now, let $\bar G$ be a solution to the above equation when $c=1$.
For a given $x>0$, consider the equation
\begin{equation*}
    \begin{dcases}
        \frac{y}{(1-y)^3} = x \,, \\
        0<y <1  \,.
    \end{dcases}
\end{equation*}
Denote  by $\phi(y) = x(1-y)^3 - y$ for $y\in [0,1]$. 
As $\phi'(y) = -3x(1-y)^2 - 1 < 0$, $\phi(y)$ is strictly decreasing on $[0,1]$. Since $\phi(0) = x$ and $\phi(1) = -1$, there exists a unique $y=y_x \in (0,1)$ such that $\phi(y_x) = 0$. 
Letting $z = 1- y$, we have that
\begin{equation*}
    z^3 + \frac{1}{x} z - \frac{1}{x} = 0 \,.
\end{equation*}
The Cardano formula says that the real root $\bar z_x \in (0,1)$ of this equation is given by 
\begin{equation*}
    \bar z_x = \frac{1}{\sqrt{x}} \brak[\Big]{ \paren[\Big]{\frac{\sqrt{x+x_0} + \sqrt{x}}{2}  }^{1/3} - \paren[\Big]{\frac{\sqrt{x+x_0} - \sqrt{x}}{2}  }^{1/3} }   \,,
\end{equation*}
where $x_0 = \frac{4}{27}$.
This implies, by definition of $ \bar G$,
\begin{equation*}
   \partial_x \bar F(x) =  \bar z_x, \quad \text{ and } \quad \bar G(x) = 1 - \partial_x \bar F(x) \,.
\end{equation*}
This, in fact, shows that $\bar F$ solves the equation~\eqref{HJ-s}. 

We now deal with general $c>0$.
To prove the scaling property, for $C>0$ denote $F_C(x) = \frac{1}{C} F(Cx)$, we have
\begin{equation*}
    G_C(x) \defeq 1 - \partial_x F_C(x) = 1 - \partial_x F(Cx) = G(Cx) \,.
\end{equation*}
Using equation~\eqref{e:transcendental} for $C = 1/c$, we get
\begin{equation*}
    \frac{G_C(x)}{(1- G_C(x))^3} = x \,.
\end{equation*}
By uniqueness of the solution to the above equation that satisfies $0< G_C <1$, we deduce that $G_C = \bar G$.
 Therefore, $F_C= \bar F$.
 Thus, for each solution $F$ of equation~\eqref{HJ-s} satisfying~\eqref{eq:A1}, there exists a $c>0$ so that $F(x) = \frac{1}{c}\bar F(cx)$.
\end{proof}
\begin{remark}
    As noted in~\cite{DegondLiuEA17}, $-\frac{4}{27}$ is the minimum value of the function $u\mapsto \frac{u}{(1-u)^3}$ at $u = -\frac{1}{2}$.
    
    \smallskip
    
    Besides, we have a bit further understanding of $\bar F$ as following.
    Denote by
    \[
    \alpha(x)=\paren[\Big]{\frac{\sqrt{x+x_0} + \sqrt{x}}{2}  }^{1/3}, \qquad
    \beta(x)=\paren[\Big]{\frac{\sqrt{x+x_0} - \sqrt{x}}{2}  }^{1/3}.
    \]
    Then, $\alpha(x)\beta(x)=\frac{1}{3}$, and
    \begin{align*}
    \partial_x \bar F(x) 
    &= \frac{\alpha(x) - \beta(x)}{\sqrt{x}}  = \frac{1}{\sqrt{x}} \cdot \frac{\alpha(x)^3 - \beta(x)^3}{\alpha(x)^2+\beta(x)^2 +\alpha(x) \beta(x)} \\
    &= \frac{1}{\alpha(x)^2+\beta(x)^2 +\alpha(x) \beta(x)}
    = \frac{1}{\alpha(x)^2+\frac{1}{9\alpha(x)^2} +\frac{1}{3}}.
    \end{align*}
    This gives us some further qualitative properties of $\partial_x \bar F(x)$.
    Indeed, it is clear that $\partial_x \bar F(0)=1$ as $\alpha(0)=\beta(0)=\frac{1}{\sqrt{3}}$.
    Note, also that as $z \mapsto \frac{1}{z^2 + \frac{1}{9z^2} + \frac{1}{3}}$ is strictly decreasing for $z > \frac{1}{\sqrt{3}}$ and
    \begin{equation} \label{eq:rateofdecay}
        \lim_{x\to \infty} x^{1/3}\partial_x \bar F(x) = 1 \,,
    \end{equation}
    $x\mapsto \partial_x \bar F(x)$ is strictly decreasing and $\partial_x \bar F(x)$ decays like $x^{-1/3}$ as $x\to \infty$.
    This implies  that $\bar F$ is sublinear as
    \[
    \lim_{x\to \infty} \frac{\bar F(x)}{x} = \lim_{x \to \infty} \partial_x \bar F(x)=0 \,.
    \]
\end{remark}

Next is another characterization of solutions to \eqref{HJ-s}.

\smallskip

\begin{proposition}\label{prop:characterization2}
Let $F$ be a solution to \eqref{HJ-s} such that F satisfies
\begin{equation}\label{eq:A2}
\begin{cases}
F \text{ is concave on $[0,\infty)$}, \\
0 \leq \partial_x F(x) <1 \qquad \text{ for a.e. $x\in (0,\infty)$.}
\end{cases}
\end{equation}
Then, either $F \equiv 0$ or there exists $c>0$ such that
\[
F(x) = \frac{1}{c} \bar F(cx) \qquad \text{ for all } x\geq 0.
\]
\end{proposition}

\begin{proof}
If there exists $z\in (0,\infty)$ such that $\partial_x F(z)=0$, then by the concavity of $F$, we imply that $\partial_x F(x)=0$ for all $x \geq z$.
Use this relation in \eqref{HJ-s} to get further that $F(x)=0$ for all $x\geq z$. Thus, $F \equiv 0$.

\medskip

We now only need to consider the case that $0<\partial_x F(x)<1$ for a.e. $x\in (0,\infty)$.
As $F$ is concave, $x \mapsto \partial_x F(x)$ is decreasing whenever $\partial_x F(x)$ is defined.
Let us first show that $F \in C^1((0,\infty))$.
If this is not the case, then there exists $z\in (0,\infty)$ such that
\[
\lim_{x \to z^-} \partial_x F(x) = a > b = \lim_{x \to z^+} \partial_x F(x)
\]
for some $0<b<a<1$.
By using \eqref{HJ-s} at differentiable points $x$ of $F$ and let $x \to z^-$, $x \to z^+$, respectively, we yield
\[
\frac{1}{2}(a-1)(a-2) = \frac{1}{2}(b-1)(b-2) = 1 - \frac{F(z)}{z},
\]
which is absurd.
Thus, $F \in C^1((0,\infty))$, and of course, $0 \leq F(x)<x$ for $x>0$.

We next show that $F \in C^2((0,\infty))$.
Equation \eqref{HJ-s} can be rewritten as
\[
(\partial_x F)^2  - 3 \partial_x F  + 2 \frac{F}{x}=0,
\]
which is a quadratic equation in terms of $\partial_x F$ under the condition that $\partial_x F \in (0,1)$.
Thus,
\[
\partial_x F(x) = \frac{3 - \sqrt{9-8\frac{F(x)}{x}}}{2}
\]
As $F \in C^1((0,\infty))$, we deduce that the right hand side of the above is $C^1$ as well, which means that $\partial_x F \in C^1((0,\infty))$.
In fact, by induction, we are able to yield that $F \in C^\infty((0,\infty))$.
We then see that $F$ satisfies \eqref{eq:A1}, and use Proposition \ref{prop:characterization1} to conclude.
\end{proof}

We are ready for the proof of our main result in this section.

\smallskip

\begin{proof}{(\bf Proof of Theorem \ref{thm:characterization3})}
Firstly, we have that
\[
\frac{1}{2}(\partial_x F - 1)(\partial_x F - 2) =1 - \frac{F}{x} \leq 1 \qquad \text{ in } (0,\infty),
\]
which yields that $F$ is Lipschitz on $[0,\infty)$, and $0 \leq \partial_x F(x) \leq 3$ for a.e. $x\in [0,\infty)$. 
At each differentiable point $x$ of $F(x)$, $\partial_x F(x)$ satisfies a quadratic equation
\[
(\partial_x F)^2  - 3 \partial_x F  + 2 \frac{F}{x}=0,
\]
which means that
\[
\partial_x F(x) = \frac{3 \pm \sqrt{9-8\frac{F(x)}{x}}}{2}.
\]
We claim first that
\begin{equation}\label{eq:claim1}
\partial_x F(x) = \frac{3 - \sqrt{9-8\frac{F(x)}{x}}}{2} \qquad \text{ for a.e. } x\in [0,\infty).
\end{equation}
Assume otherwise that \eqref{eq:claim1} does not hold true, then there exists $z\in (0,\infty)$ such that $F$ is differentiable at $z$ and
\[
\partial_x F(z) = \frac{3 + \sqrt{9-8\frac{F(z)}{z}}}{2} \geq 2.
\]
On the other hand, by the comparison principle, $0\leq F(x)\leq x$ for all $x \in [0,z]$, and $F$ is Lipschitz, we are able to find $y \in (0,z)$ such that $F$ is differentiable at $y$ and $\partial_x F(y) \leq 1$.
Set
\[
\phi(x) = F(x) - \frac{3}{2}x \qquad \text{ for } x \in [y,z].
\]
Of course, $\phi$ obtains its minimum at some point $\bar x \in [y,z]$.
It is not hard to see that $\bar x \neq y$ and $\bar x \neq z$ as 
\[
\partial_x \phi(y) \leq - \frac{1}{2} \quad \text{ and } \quad
\partial_x \phi(z) \geq \frac{1}{2}.
\]
So, $\bar x \in (y,z)$, which means that $F(x)- \frac{3}{2}x$ has a local minimum at $\bar x$.
By the viscosity supersolution test to \eqref{HJ-s}, we yield that
\[
\frac{1}{2}\left(\frac{3}{2}-1\right)\left(\frac{3}{2}-2\right) + \frac{F(\bar x)}{\bar x} - 1 = \frac{F(\bar x)}{\bar x} -\frac{9}{8} \geq 0,
\]
which is absurd.

\medskip

Thus, \eqref{eq:claim1} holds.
It is important noting that the right hand side of \eqref{eq:claim1} is continuous in $x$.
By the fundamental theorem of calculus, we are able to write
\[
 F(x) = \int_0^x \frac{1}{2}\left(3 - \sqrt{9-8\frac{F(y)}{y}}\right)\,dy,
\]
and hence, $F \in C^1((0,\infty))$ and  \eqref{eq:claim1} holds true for all $x\in (0,\infty)$.
In fact, we have $F \in C^\infty((0,\infty))$.
By using the fact that $0\leq F(x) \leq x$, we imply further that
\[
0 \leq \partial_x F(x) \leq 1,
\]
and $\partial_x F(x)=0$ if and only if $F(x)=0$.
In particular, $F$ is nondecreasing.
It only remains to prove that $\partial_x F(x) <1$ for all $x \geq 0$.

\medskip

If $F \equiv 0$, then there is nothing to consider.
We hence only need to focus on the case $F \neq 0$.
Since $F$ is also sublinear, we are able to find $z \geq 0$ such that
\[
0<F(x)<x \qquad \text{ for all } x > z.
\]
Use this in \eqref{eq:claim1} to yield that
\begin{equation}\label{eq:claim2}
0 < \partial_x F(x) <1  \qquad \text{ for all } x > z.
\end{equation}
Thanks to \eqref{eq:claim2}, we are able to repeat the first part of the proof of Proposition \ref{prop:characterization1} to have that, for $G=1-\partial_x F$,
\[
\frac{G(x)}{(1-G(x))^3} = cx \qquad \text{ for all } x > z.
\]
Here, $c>0$ is some fixed constant.
Without loss of generality, we assume $c=1$.
By repeating the later part of the proof of Proposition \ref{prop:characterization1},
$G(x) = \bar G(x)$ for $x > z$, and hence,
\[
\partial_x F(x) = \partial_x \bar F(x) = \frac{1}{\alpha(x)^2+\beta(x)^2 +\alpha(x) \beta(x)} \qquad \text{ for all } x > z.
\]
We finally claim that
\begin{equation}\label{eq:claim3}
\partial_x F(x) = \partial_x \bar F(x)  \qquad \text{ for } x > 0,
\end{equation}
that is, we can let $z=0$ in \eqref{eq:claim2}.
Indeed, if this is not the case, then there is $\bar z>0$ such that \eqref{eq:claim2} holds for  $z=\bar z$, and $\partial_x F(\bar z) \in \{0,1\}$.
On the other hand, 
\[
\partial_x F(\bar z) =\lim_{x \to\bar z^+} \partial_x F(x) =\lim_{x \to\bar z^+} \partial_x \bar F(x)  = \partial_x \bar F(\bar z) \in (0,1),
\]
which is absurd.
Thus, \eqref{eq:claim3} holds, and $F= \bar F$.
The proof is complete.
 \end{proof}

\section{Large time behavior of \eqref{HJ-1}}\label{sec:large-time}

Let $F$ be the viscosity solution to \eqref{HJ-1}.
Under assumption \eqref{con:A1}, we have that $F$ is sublinear in $x$, globally Lipschitz, and
\begin{equation}\label{eq:bound on F_x}
0 \leq \partial_x F(x,t) \leq 1 \quad \text{ for a.e. } (x,t) \in [0,\infty)^2.
\end{equation}
We refer the reader to \cite[Lemma 3.1]{TranVan2019} for a proof of this observation.
This assumption \eqref{con:A1} is, however, not enough to obtain large time behavior of the viscosity solution $F(x,t)$ to \eqref{HJ-1}.
It turns out that the behavior of $F_0(x)$ for $x \to \infty$ does play an important role in determining the behavior of $F(x,t)$ as $t\to \infty$.
If we look into the behavior of the stationary solution $\bar F$, then we see that by~\eqref{eq:rateofdecay},
\[
\lim_{x\to \infty}\frac{\bar F(x)}{x^{2/3}} = 
\lim_{x\to \infty}\frac{\partial_x \bar F(x)}{\frac{2}{3}x^{-1/3}} = \frac{3}{2}.
\]

This gives us some intuition that $x^{2/3}$ represents a critical growth of initial condition, and the large time behavior of $F$ depends crucially on the relative growth of $F_0$ compared to this critical growth.

\subsection{Initial condition with subcritical growth} \label{sec:subcritical}
In this subsection, we study the viscosity solution with subcritical initial data.

\smallskip

We first recall the representation of the viscosity solution to \eqref{HJ-1} from optimal control theory. 
For $(x,t) \in [0,\infty)^2$, denote by
\begin{multline}\label{func:value}
V(x,t) = \inf_{\substack{\gamma \in {\rm AC}([0,t],[0,\infty)) \\ \gamma(0)=x}}
\Big\{
\int_{0}^{t\land\tau_x}\frac{1}{2}e^{-\int_0^s\frac{d\lambda}{\gamma(\lambda)}}\,
\left(-\dot{\gamma}(s)+\frac{3}{2}\right)^2\,ds \, +\\
e^{-\int_0^{t\land\tau_x}\frac{d\lambda}{\gamma(\lambda)}}\,G(\gamma(t\land\tau_x),t-t\land\tau_x)
\Big\}. 
\end{multline}
Here, ${\rm AC}([0,t],[0,\infty))$ is the space of absolutely continuous curves mapping from $[0,t]$ to $[0,\infty)$.
Besides, $t\land\tau_x = \min\{t,\tau_x\}$, and
\begin{align*}
&\tau_x:=\inf\{s\ge0\,:\, \gamma(s)=0\}\le\infty, \\
& 
G(x,t):=
\left\{
\begin{array}{ll}
0 \quad & \text{ if }  x=0, \\
F_0(x) \quad & \text{ if }  t=0.  
\end{array}
\right. 
\end{align*}

Bellman's principle of optimality claims that an optimal policy has the property that whatever the initial state is,  the remaining decisions must constitute an optimal policy with regard to the state resulting from the first decision. 
Following to this principle, we have the following Dynamical Programming Principle. 

\smallskip

\begin{proposition}[Dynamical Programming Principle]\label{prop:DPP}
For $(x,t) \in [0,\infty)^2$ and $h>0$, we have 
\begin{align*}
&V(x,t+h)\\
=&\,
\inf_{\gamma(0)=x}
\Big\{
\int_0^{h\land\tau_x}\frac{1}{2}e^{-\int_0^s\frac{d\lambda}{\gamma(\lambda)}}\left(-\dot{\gamma}(s)+\frac{3}{2}\right)^2\,ds
+
\mathbf{1}_{\{h<\tau_x\}}e^{-\int_0^h\frac{d\lambda}{\gamma(\lambda)}}V(\gamma(h),t)\\
&\qquad \qquad \qquad \qquad \qquad \qquad \qquad \quad+
\mathbf{1}_{\{h\ge\tau_x\}}e^{-\int_0^{\tau_x}\frac{d\lambda}{\gamma(\lambda)}}V(\gamma(\tau_x),t-\tau_x)
\Big\} \\
=&\,
\inf_{\gamma(0)=x}
\Big\{
\int_0^{h\land\tau_x}\frac{1}{2}e^{-\int_0^s\frac{d\lambda}{\gamma(\lambda)}}\left(-\dot{\gamma}(s)+\frac{3}{2}\right)^2\,ds
+
\mathbf{1}_{\{h<\tau_x\}}e^{-\int_0^h\frac{d\lambda}{\gamma(\lambda)}}V(\gamma(h),t)\Big\}.
\end{align*}
Here,
$\mathbf{1}_{\{h<\tau_x\}}=1$ and $\mathbf{1}_{\{h\ge\tau_x\}}=0$ if $h<\tau_x$, 
and $\mathbf{1}_{\{h<\tau_x\}}=0$ and $\mathbf{1}_{\{h\ge\tau_x\}}=1$ if $h\ge \tau_x$. 
\end{proposition}
The proof of Proposition \ref{prop:DPP} is rather standard by using the usual arguments in the optimal control theory (see \cite{L, BCD, Tran19} for instance). 
By Proposition \ref{prop:DPP} and classical techniques in the theory of viscosity solutions, we have the following result.

\smallskip

\begin{proposition}\label{prop:value}
Assume \eqref{con:A1}. 
Let $F$ be the unique viscosity solution to \eqref{HJ-1}.
Then, $F=V$ on $[0,\infty)^2$.
\end{proposition}

We skip the proofs of Propositions \ref{prop:DPP} and \ref{prop:value}, and we refer the readers to \cite{L, BCD, Tran19} for details.
By using Proposition \ref{prop:value}, we prove Theorem \ref{thm:large time2}. 

\smallskip

\begin{proof}{(\bf Proof of Theorem~\ref{thm:large time2} )}
Fix $x>0$.
Let $\gamma(s):=x+\frac{3}{2}s$ for $s \geq 0$. 
Then, $\tau_x=\infty$. 
By formula \eqref{func:value}, we see that
\[
0\le F(x,t)
\le 
e^{-\int_0^{t}\frac{d\lambda}{x+\frac{3\lambda}{2}}}F_0\left(x+\frac{3t}{2}\right).
\]
Noting that 
\[
\int_0^{t}\frac{d\lambda}{x+\frac{3\lambda}{2}}=\frac{2}{3}\log\left(\frac{3t}{2x}+1\right),  
\]
we have 
\[
e^{-\int_0^{t}\frac{d\lambda}{x+\frac{3\lambda}{2}}}
=\left(\frac{3t}{2x}+1\right)^{-2/3},  
\]
which implies 
\begin{equation}\label{est:1}
0\le F(x,t)
\le \left(\frac{3t}{2x}+1\right)^{-2/3}F_0\left(x+\frac{3t}{2}\right)
=\frac{F_0\left(x+\frac{3t}{2}\right)}{\left(x+\frac{3t}{2}\right)^{2/3}}\, x^{2/3}.  
\end{equation}
We now use \eqref{con:A2} to conclude right away that 
\[
F(x,t) \to 0 \qquad \text{ locally uniformly for $x\in[0,\infty)$.} 
\]   
\end{proof}

\smallskip

\begin{remark}
    It is worth noting that in the proof of Theorem \ref{thm:large time2}, we do not need to require fully condition \eqref{con:A1}. 
    More precisely, condition \eqref{con:A1} can be replaced by a much weaker one
    \begin{equation}\label{con:A1-p}
    F_0 \in \Lip([0,\infty)), \ 0 \leq F_0(x) \leq x.
    \end{equation}
\end{remark}

\subsection{Initial condition with supercritical growth} \label{sec:supercritical}
In this subsection, we study the solution when the initial data is supercritical.

\smallskip

\begin{proof}{(\bf Proof of Theorem~\ref{thm:large time3})}
Fix $y \in (0,\infty)$.
We note that, by backward characteristics (or by the optimal control formulation), an optimal path $X:[0,t] \to [0,\infty)$ with $X(t)=y$ satisfies the Hamiltonian system
 \begin{equation}\label{HS}
  \begin{cases}
  \dot X = \partial_p H = P(s) - \frac{3}{2}\,,\\
  \dot P = -\partial_x H - (\partial_z H)P = \frac{Z(s)}{X(s)^2}- \frac{P(s)}{X(s)}\,,\\
  \dot Z = P\cdot \partial_pH - H = \frac{P(s)^2}{2} - \frac{Z(s)}{X(s)}\,.
  \end{cases}
  \end{equation}
Here, $X(0)=x$ for some $x\geq 0$.
Moreover, $F$ is differentiable at $(X(s),s)$,
$P(s)=\partial_x F(X(s),s)$, and $Z(s)=F(X(s),s)$ for $0\leq s <t$.
There can be more than one optimal paths (backward characteristics), in which case $F$ might not be differentiable at $(y,t)=(X(t),t)$.
In any case, thanks to \eqref{eq:bound on F_x}, we have that $0\leq P(s) \leq 1$ for $0\leq s <t$.
Thus, 
\[
-\frac{3}{2} \leq \dot X(s) \leq -\frac{1}{2} \quad \text{ for all } s \in (0,t),
\]
which means that
\[
X(0)=x \geq y + \frac{t}{2}.
\]
As $t \to \infty$, $X(0) \to \infty$.
Therefore, the information of $F_0$ at $+\infty$ determine the behavior of $F(y,t)$ as $t \to \infty$.

\medskip

Thanks to \eqref{con:A3}, for any fixed $c>0$, 
there exists $r_c>0$ such that,
\[
F_0(x) \geq \frac{1}{c}\bar F(cx) \quad \text{ for all } x>r_c.
\]
Denote by $F_{0c}(x) = \min\{F_0(x),\frac{1}{c}\bar F(cx)\}$ for $x\geq 0$.
Let $F_c$ be the solution to \eqref{HJ-1} with initial condition $F_{0c}$.
Since $\frac{1}{c}\bar F(cx)$ is a stationary solution to \eqref{HJ-1}, and $F_{0c}(x)=\frac{1}{c}\bar F(cx)$ on $[r_c,\infty)$,
 we have that
\[
F(y,t) \geq F_c(y,t) = \frac{1}{c}\bar F(cy) \quad \text{ for all $t>2|r_c-y|$.}
\]
Thus, it is clear that
\begin{equation}\label{lower-bound-Fc}
\liminf_{t \to \infty} F(y,t) \geq \frac{1}{c}\bar F(cy) \quad \text{ locally uniformly for } y \in [0,\infty).
\end{equation}
The above \eqref{lower-bound-Fc} holds true for every $c>0$.
Note further that
\[
\lim_{c \to 0^+} \frac{1}{c}\bar F(cy)=\lim_{c \to 0^+} \frac{\bar F(cy)- \bar F(0)}{c}=\partial_x \bar F(0) y = y,
\]
which gives that 
\[
\liminf_{t \to \infty} F(y,t) \geq y \quad \text{ locally uniformly for } y \in [0,\infty).
\]
The conclusion follows.
\end{proof}

\subsection{Initial condition with critical growth} \label{sec:critical}
In this subsection, we study the solution with critical initial data.
We first argue that \eqref{con:A4} can be interpreted in a more intuitive way as following.
Let $c=\frac{27}{8\delta^3}$.
Then,
\begin{equation}\label{lim-c}
\lim_{x\to \infty}\frac{1}{c x^{2/3}} \bar F(cx)
=\lim_{x\to \infty}\frac{1}{c^{1/3}} \frac{\bar F(cx)}{(cx)^{2/3}} = \frac{3}{2c^{1/3}}= \frac{3}{2} \times \frac{2\delta}{3} = \delta.
\end{equation}
Thus,
\[
\lim_{x\to \infty}\frac{\bar F(cx)}{c F_0(x)} =1,
\]
which implies that \eqref{con:A4} is equivalent to the following condition
\begin{equation}\label{con:A5}
\frac{1}{c} \bar F(cx) -  h(x) \leq F_0(x) \leq \frac{1}{c} \bar F(cx) + h(x) \quad \text{ for all } x\geq 0,
\end{equation}
for $c=\frac{27}{8\delta^3}$.
Here, $h: [0,\infty) \to [0,\infty)$ is a function satisfying that
\[
\lim_{x \to \infty} \frac{h(x)}{x^{2/3}}=0.
\]

The idea of this proof is quite close to that of Theorem \ref{thm:large time3}, so we will not include all the details here.

\smallskip

\begin{proof}{(\bf Proof of Theorem~\ref{thm:large time4})}
We first note that \eqref{con:A5} holds.

\smallskip

Fix $y \in (0,\infty)$.
We note that, by backward characteristics (or by the optimal control formulation), an optimal path $X:[0,t] \to [0,\infty)$ with $X(t)=y$ satisfies the Hamiltonian system \eqref{HS}.
Here, $X(0)=x$ for some $x\geq 0$, $P(s)=\partial_x F(X(s),s)$, and $Z(s) = F(X(s),s)$ for $0\leq s <t$.
There can be more than one optimal paths (backward characteristics), in which case $F$ might not be differentiable at $(y,t)=(X(t),t)$.
By the same argument as in the proof of Theorem \ref{thm:large time3},
\[
X(0)=x \geq y + \frac{t}{2}.
\]
As $t \to \infty$, $X(0) \to \infty$.
Thus, the information of $F_0$ at $+\infty$ determine the behavior of $F(y,t)$ as $t \to \infty$.

Fix  $d>c$. Thanks to \eqref{lim-c} for $d=c$ and \eqref{con:A5}, there exists $x_d>0$ such that for $x\geq x_d$,
\[
F_0(x) \geq \frac{1}{d}\bar F(dx).
\]
Since $\frac{1}{d}\bar F(dx)$ is a stationary solution to \eqref{HJ} and only information at $+\infty$ of $F_0$ matters in the behavior of $F(y,t)$, it is clear that
\begin{equation*}
\liminf_{t \to \infty} F(y,t) \geq \frac{1}{d}\bar F(dy).
\end{equation*}
The above \eqref{lower-bound-Fc} holds true for every $d>c$, which gives further that
\begin{equation}\label{lower-bound-Fy}
\liminf_{t \to \infty} F(y,t) \geq \lim_{d \to c^+}\frac{1}{d}\bar F(dy)= \frac{1}{c}\bar F(cy).
\end{equation}
To get the upper bound, we perform the analysis in a similar way for $d \in (0,c)$ by noting that, for $x\gg 1$,
\[
F_0(x) \leq \frac{1}{d}\bar F(dx),
\]
and hence,
\begin{equation}\label{upper-bound-Fy}
\limsup_{t \to \infty} F(y,t) \leq \lim_{d \to c^-}\frac{1}{d}\bar F(dy)= \frac{1}{c}\bar F(cy).
\end{equation}
Combine \eqref{lower-bound-Fy} and \eqref{upper-bound-Fy} to conclude.
\end{proof}

\section{A non-convergence result} \label{sec:nonconvergence}

In this section, we give the proof of Theorem~\ref{thm:nonconvergence}. The meaning of this theorem is that if we do not have \eqref{con:A4} (or equivalently, \eqref{con:A5}), then large time behavior might not hold.
In other words, our claim is that the requirements in Theorem \ref{thm:large time4} are optimal if one wants to expect large time convergence.
We start with the following elementary fact.

\smallskip

\begin{lemma} \label{lem:comparison}
Let $c_1 , c_2 \in (0,\infty)$ be such that $c_1<c_2$. 
We have 
\begin{equation*}
    F_1(x) = \frac{1}{c_1} \bar F(c_1 x) > \frac{1}{c_2} \bar F(c_2 x) = F_2(x) \quad \text{ for all } x>0\,.
\end{equation*}
\end{lemma}

\begin{proof}
    We have that for every $x>0$,
    \begin{equation*}
        F_1'(x) = \bar F'(c_1 x) > \bar F'(c_2 x) = F_2'(x) >0 \,.
    \end{equation*}
    As $F_1(0) = F_2(0) = 0$, we conclude that
    \begin{equation*}
        F_1(x) > F_2(x) \quad \text{ for all } x>0\,,
    \end{equation*}
    as desired.
\end{proof}

\smallskip

We are now ready to explicitly construct an initial data $F_0$ to prove Theorem~\ref{thm:nonconvergence}.

\smallskip

\begin{proposition} \label{prop:nonconvergence}
    Fix $x_0 >0$ and let $F_1$ and $F_2$ be as in Lemma~\ref{lem:comparison}.
    There exist unbounded increasing sequences of positive real numbers $\set{a_i}_{i=0}^\infty$, and $\set{t_i}_{i=1}^\infty$, where $a_0 =0$, such that for 
    \begin{equation*}
        F_0(x) = \begin{dcases}
            F_1(x) &\text{ if } x \in [a_{4i}, a_{4i+1}], \\
            F_1(a_{4i+1})  &\text{ if } x \in[a_{4i+1}, a_{4i+2}] \,,\\
            F_2(x) &\text{ if } x \in[a_{4i+2}, a_{4i+3}] \,, \\ 
            F_2(a_{4i+3})+ F_2'(a_{4i+3}) (x- a_{4i+3})    &\text{ if } x \in[a_{4i+1}, a_{4i+2}] \,,\\
        \end{dcases}
    \end{equation*}
    we have $F_0 \in \Lip([0,\infty))$ satisfies \eqref{con:A1} in the a.e. sense, and 
    \begin{equation*}
        \lim_{i\to \infty } F(x_0,t_{2i+1}) = F_1(x_0) > F_2(x_0) = \lim_{i\to\infty} F(x_0, t_{2i}) \,.
    \end{equation*}
\end{proposition}
Let $F_0$ be as above.
It is worth noting that, by the computation at the beginning of  Section~\ref{sec:critical},
\[
\frac{3}{2 c_2^{1/3}}=\liminf_{x\to \infty} \frac{F_0(x)}{x^{2/3}} <\limsup_{x\to \infty} \frac{F_0(x)}{x^{2/3}} =\frac{3}{2 c_1^{1/3}}.
\]

\begin{proof}
    The key observation here is that the characteristics $X$ (defined in the proof of Theorem~\ref{thm:large time4}) has bounded slopes, i.e.,
    \begin{equation} \label{ine:slopeBounds}
       -\frac{3}{2} \leq \dot X \leq -\frac{1}{2} \,. 
    \end{equation}
    
    \begin{figure}[hbt!]
        \centering
        \includegraphics[width=11cm]{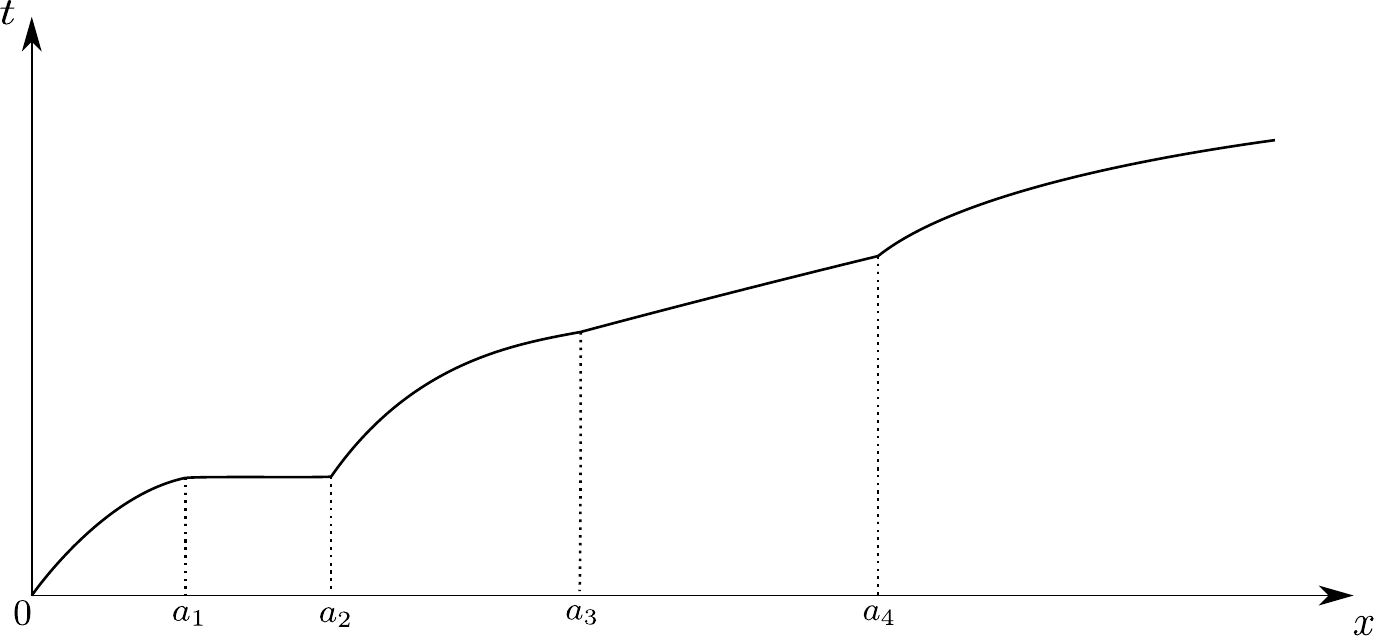}
        \caption{Initial data $F_0$}
    \end{figure}
    
    \begin{figure}[hbt!]
        \centering
        \includegraphics[width=11cm]{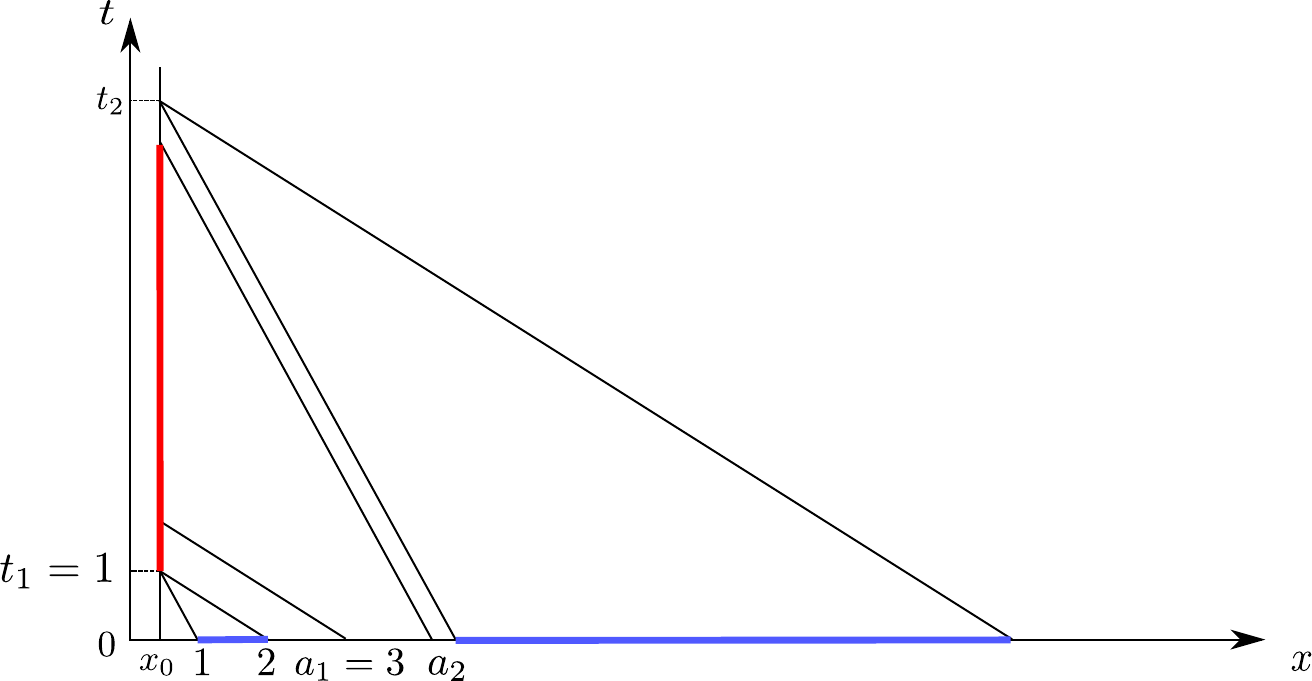}
        \caption{Domains of dependence (blue) and range of influence (red)}
    \end{figure}
   \noindent For simplicity, we first fix $x_0 = \frac{1}{2}$ although the argument works for any $x_0>0$. 
   The construction of $a_i$'s and $t_i$'s is as follows. 
   
   \smallskip
   
   \emph{Step 1.}  By~\eqref{ine:slopeBounds}, we have that the domain of dependence of $(x_0,t_0) \in (0,\infty)^2$ is  
    \begin{equation} \label{domainOfDependence}
        \Delta(x_0,t_0)=\left \{ x\,:\, x_0 + \frac{1}{2} t_0 \leq x \leq  x_0+ \frac{3}{2} t_0   \right\} \, .
    \end{equation}
    That is, $F(x_0,t_0)$ is determined by information of $F_0$ on $\Delta(x_0,t_0)$.
     On the other hand, given $x > x_0$, the range of influence when $X(x,t) = x_0$ is
    \begin{equation} \label{rangeOfInfluence}
       I(x)=\left\{t\,:\, \frac{2}{3}(x-x_0) \leq t \leq 2(x - x_0) \right\}  \,.
    \end{equation}
    This means that $F_0(x)$ might be able to influence $F(x_0,t)$ for $t \in I(x)$.
    
    \smallskip
    
    Recall $x_0=\frac{1}{2}$, and let $t_1  = 1$,  $a_1 = 3$. 
    By noting that $\Delta(x_0,t_1) = [1,2] \subset [0,3]$ by~\eqref{domainOfDependence}, we get that
    \begin{equation*} 
        F(x_0, t_1) = F_1(x_0) \,.
    \end{equation*}
    
    \emph{Step 2.} 
    Since $F_1(a_1)>F_2(a_1)$ and $F_2$ is strictly increasing, there exists a unique $a_2>a_1$ such that $F_2(a_2) = F_1(a_1)$.
    By~\eqref{rangeOfInfluence}, the range of influence of $[a_1,a_2]$ is
    \begin{equation*}
        \frac{2}{3}( a_1 - \frac{1}{2}) \leq t \leq 2( a_2 - \frac{1}{2}) \,.
    \end{equation*}
    Then, the domain of dependence of $\set{x_0} \times [2(a_2 - 1/2), 2(a_2 - 1/2) + 1]$ is
    \begin{equation*}
        x_0 + (a_2 - \frac{1}{2}) \leq x \leq x_0 + 3(a_2 - \frac{1}{2}) + \frac{3}{2}  \,.
    \end{equation*}
    
    Then, let $t_2 = 2(a_2 - 1/2) + 1/2$ and $a_3 = 3(a_2 - 1/2) + 2$.
    By construction, we have that the domain of dependence of $(x_0, t_2)$ is contained in $(a_2, a_3)$ and therefore,
    \begin{equation*}
        F(x_0, t_2) = F_2(x_0) \,.
    \end{equation*}
    
    Let $a_4 > a_3$ so that $F_2(a_3)+F_2'(a_3) (a_4 - a_3)   = F_1(a_4)$ ($a_4$ exists because $F_2$ is sublinear). Then, we pick $t_3$ and $a_5$ the same way with picking $t_2$ and $a_3$, i.e.,
    \begin{equation*}
        t_3  = 2(a_4 - \frac{1}{2}) + \frac{1}{2} \quad \text{ and } \quad a_5 = 3(a_4 - \frac{1}{2}) + 2 \,.
    \end{equation*}
    Reasoning as above, we conclude that
    \begin{equation*}
        F(x_0,t_3) = F_1(x_0) \,.
    \end{equation*}

    \emph{Step 3.} Repeat Step 2 indefinitely. By construction, $F_0 \in \Lip([0,\infty))$ satisfies \eqref{con:A1} in the a.e. sense, and
    \begin{equation*}
        F(x_0,t_{2i+1}) = F_1(x_0) 
        \quad \text{ and } \quad
        F(x_0,t_{2i}) = F_2(x_0) 
    \end{equation*}
    for every $i \in \N$. This implies what we want to prove.
\end{proof}

\smallskip

\begin{remark}
    We deliberately avoided the regions where shocks might occur in the above construction. However, viscosity solutions make sense for all time and still admit characteristics where there is no shocks.
\end{remark}

\smallskip

\begin{proof}{(\bf Proof of Theorem~\ref{thm:nonconvergence})} The nonconvergence result follows immediately from Proposition~\ref{prop:nonconvergence}. \end{proof}

\section*{Acknowledgement}
HM is supported by the JSPS through grants KAKENHI \#19K03580, \#19H00639, \#17KK0093, \#20H01816. 
HT is supported in part by NSF grant DMS-1664424 and NSF CAREER grant DMS-1843320.

\appendix

\section{Derivation of Hamilton-Jacobi equation \eqref{HJ}} \label{appendix}
We give the derivation of equation \eqref{HJ} here for completeness of the paper, which is taken from \cite{TranVan2019}.
In order to derive equation~\eqref{HJ}, we utilize the weak form of the C-F equation~\eqref{weak sln} with the test function $\phi^x(s) = 1 - e^{-sx}$.
By noting the important identity that \[
\phi^x(s+\hat s) - \phi^x(s) - \phi^x(\hat s)=-\phi^x(s) \phi^x(\hat s),
\]
we have
\begin{align*}
  \partial_t F(x,t) &= \frac{1}{2} \int_0^\infty \int_0^\infty ( 1 - e^{-(s + \hat s) x} - 1 + e^{-sx} -1 + e^{-\hat s x} ) s \rho(s,t) \hat \rho(\hat s, t) \, d\hat s ds \\
  & - \frac{1}{2} \int_0^\infty \int_0^s ( 1 - e^{-sx} - 1 + e^{-(s - \hat s)x} - 1 + e^{- \hat s x} ) \, d \hat s \, \rho(s,t) \, ds \\ 
  &=  -\frac{1}{2} \int_0^\infty \int_0^\infty (1 - e^{-sx})(1 - e^{-\hat s x}) s \rho(s,t) \hat s \rho(\hat s,t) \, d\hat s ds \\
  & - \frac{1}{2} \int_0^\infty ( - s - s e^{-sx} + \frac{2}{x}(1  - e^{-sx}) ) \rho(s,t) \, ds  \\
  &= - \frac{1}{2}( m_1(t) - \partial_x F(x,t))^2 + \frac{m_1(t)}{2} + \frac{ \partial_x F(x,t)}{2} - \frac{F}{x} \\
  &= - \frac{1}{2} ( m_1(t) - \partial_x F(x,t)) ( m_1(t) - \partial_x F(x,t) + 1) - \frac{F}{x} + m_1(t) \,.
\end{align*}
Equation~\eqref{HJ} follows if we assume conservation of mass, i.e.,
\begin{equation*}
    m_1(t) = m \quad \text{ for all } t \geq 0\,.
\end{equation*}
\bibliography{bibliography.bib}
\bibliographystyle{plain}

\end{document}